\documentclass{amsart}
\usepackage{amssymb,amsmath,amsfonts,amsthm,graphicx,enumerate,amscd,latexsym,curves,multibox}
\usepackage{bbm}
\usepackage{mathrsfs}
\usepackage[mathscr]{eucal}
\usepackage{epsfig,epsf,xypic,epic}

\usepackage[T1]{fontenc}
\usepackage[sc]{mathpazo}

\newtheorem{lem}{Lemma}[section]

\newtheorem{prop}{Proposition}[section]
\newtheorem{defn}{Definition}[section]
\newtheorem{nb}{Remark}[section]

\numberwithin{equation}{section}

\begin{document}

\title[The Ooguri-Vafa metric and wall-crossing]
{The Ooguri-Vafa metric, holomorphic discs and wall-crossing}
\author[K.-W. Chan]{Kwokwai Chan}
\address{Department of Mathematics, Harvard University, Cambridge, MA 02138}
\email{kwchan@math.harvard.edu}

\begin{abstract}
Recently, Gaiotto, Moore and Neitzke \cite{GMN08} proposed a new
construction of hyperk\"{a}hler metrics. In particular, they gave a
new construction of the Ooguri-Vafa metric, in which they came
across certain formulas. We interpret those formulas as
wall-crossing formulas that appear in the SYZ construction of
instanton-corrected mirror manifolds. This reveals the close
relation between the Ooguri-Vafa metric and nontrivial holomorphic
discs with boundary in special Lagrangian torus fibers.
\end{abstract}

\maketitle

\tableofcontents

\section{Introduction}

This note grew out of an attempt to understand the relationship
between the new construction of the Ooguri-Vafa metric by Gaiotto,
Moore and Neitzke \cite{GMN08} and the wall-crossing formulas which
appear in the instanton-corrected construction of mirror manifolds
in examples investigated by Auroux \cite{Auroux07}, \cite{Auroux09}.

In their recent beautiful work \cite{GMN08}, Gaiotto, Moore and
Neitzke proposed a new construction of hyperk\"{a}hler metrics on
complex integrable systems. The simplest case of this construction
reproduces the Ooguri-Vafa metric.\footnote{The original
construction by Ooguri and Vafa \cite{OV96} was done by applying the
Gibbons-Hawking ansatz; see also Gross and Wilson \cite{GW00}.} This
sheds new light on the understanding of the metric. In particular,
one is naturally lead to certain formulas which resemble the
wall-crossing formulas in Auroux's examples of the construction of
instanton-corrected mirror manifolds \cite{Auroux07},
\cite{Auroux09}.

To connect these two constructions, we will study mirror symmetry
for the Ooguri-Vafa metric from the point of view of the SYZ
Conjecture \cite{SYZ96}. Recall that the na\"{i}ve SYZ construction
of mirror manifolds, namely, dualizing special Lagrangian torus
fibrations (or so-called T-duality in physics), in general does not
give the correct complex geometry of the mirror manifold. This is
due to the presence of singular fibers and nontrivial holomorphic
discs with boundary on special Lagrangian torus fibers. So we
usually need to modify the gluing of complex charts of the mirror
manifold by disc instanton corrections, according to certain
wall-crossing formulas. In this note, we show that the formulas
which appear in Gaiotto-Moore-Neitzke's construction of the
Ooguri-Vafa metric are \textit{equivalent} to the wall crossing
formulas which appear in Auroux's constructions. In particular, we
will be able to see how the construction of the Ooguri-Vafa metric
get contributions from corrections by disc instantons, i.e.
nontrivial holomorphic discs with boundary in special Lagrangian
torus fibers.

At this point, we shall mention that the idea to use wall-crossing
formulas (or gluing formulas) to study the SYZ version of mirror
symmetry was first suggested by Kontsevich and Soibelman in
\cite{KS04} (see also \cite{KS00}). These formulas were later
generalized and applied by Gross and Siebert \cite{GS07} in their
construction of toric degenerations of Calabi-Yau manifolds and
their mirrors from affine manifolds with singularities. The
wall-crossing formulas which appear here and in Auroux's examples in
\cite{Auroux07}, \cite{Auroux09} are all special cases of the gluing
formulas used by Kontsevich-Soibelman and Gross-Siebert.

We shall also emphasis that in all cases, the wall-crossing
phenomena are of the same kind, although there are two kinds of wall
as distinguished by Kontsevich and Soibelman in \cite{KS08}. The
formulas which we are referring to here and those in Auroux's works
correspond to crossing \textit{the wall of second kind} (see p. 27
in \cite{KS08}); while the formulas in Gaiotto-Moore-Neitzke
\cite{GMN08} correspond to crossing \textit{the wall of first kind}
(see p.30 in \cite{KS08}).\footnote{In the physics literature, the
wall of first kind is called \textit{the wall of marginal
stability}.} In fact, wall-crossing formulas of first kind, which
describe the jumping behavior of \textit{numerical Donaldson-Thomas
invariants}, play a key role in Gaiotto-Moore-Neitzke's construction
of more general hyperk\"{a}hler metrics. However, in the Ooguri-Vafa
case, the numerical Donaldson-Thomas invariants do not jump (see
Remark~\ref{rmk3.2}), so the wall-crossing formula of first kind is
trivial.\footnote{In the language of Gross-Siebert \cite{GS07},
there are no \textit{scattering} in the Ooguri-Vafa case, since
there is only one singular point, and hence one wall, in the base
affine manifold.} It is interesting to understand the relationship
between wall-crossing formulas of first kind and the construction of
instanton-corrected mirror manifolds. In particular, it is desirable
to know what the numerical Donaldson-Thomas invariants are counting.

The rest of this note is organized as follows. In the next section,
we briefly review Gaiotto-Moore-Neitzke's new construction of the
Ooguri-Vafa metric. In Section~\ref{sec3}, we proceed to explain why
the wall-crossing formulas which appear in Section~\ref{sec2} can be
interpreted as wall-crossing formulas which appear in Auroux's
construction of instanton-corrected mirror manifolds.\\

\noindent\textbf{Acknowledgements.} I am very grateful to Prof.
Shing-Tung Yau for suggesting this problem, to Andy Neitzke for
carefully explaining their constructions, and to Prof. Yan Soibelman
for answering my questions through emails and pointing out many
inaccuracies in earlier versions of this article. I would also like
to thank Prof. Denis Auroux, Prof. Mark Gross, Prof. Naichung Conan
Leung and Prof. Eric Zaslow for many useful discussions and help.
This research was supported by Harvard University and the Croucher
Foundation Fellowship.

\section{The new construction of the Ooguri-Vafa metric by Gaiotto-Moore-Neitzke
\cite{GMN08}}\label{sec2}

This section is a very brief review of the construction of the
Ooguri-Vafa metric using the method proposed by Gaiotto, Moore and
Neitzke in their recent work \cite{GMN08}. For more details and
construction of hyperk\"{a}hler metrics on general complex
integrable systems, we refer the reader to the original paper
\cite{GMN08}.\\

Let $B=\{b\in\mathbb{C}:|b|<r\}$ be the open disc centered at the
origin with radius $r>0$, and let $B'=B\setminus\{0\}$. Suppose that
$\psi:M\rightarrow B$ is an elliptic fibration with a type $I_1$
singular fiber over $0\in B$. Consider the local system
$\Gamma=R^1\psi_*\mathbb{Z}\rightarrow B$, the generic fiber of
which is given by $\Gamma_b\cong H^1(M_b,\mathbb{Z})$, where
$M_b=\psi^{-1}(b)$ is the fiber over $b\in B'$. The monodromy of
$\Gamma$ around $0\in B$ is nontrivial and given by
\begin{eqnarray*}
\gamma_e(b) & \mapsto & \gamma_e(b),\\
\gamma_m(b) & \mapsto & \gamma_m(b)+\gamma_e(b),
\end{eqnarray*}
where $\{\gamma_e(b),\gamma_m(b)\}$ is a symplectic basis of
$H^1(M_b,\mathbb{Z})$.\footnote{The subscripts "e" and "m" stand for
"electric" and "magnetic" respectively, and $\Gamma$ is called the
charge lattice in \cite{GMN08}.} This basis extends to local
sections $\gamma_e,\gamma_m$ of $\Gamma$ over a small enough open
subset $U\subset B'$. Since
$M_U=\Gamma_U^\vee\otimes_\mathbb{Z}(\mathbb{R}/2\pi\mathbb{Z})$,
the sections $\gamma_e,\gamma_m$ define local fiber coordinates
$$\theta_e,\theta_m:M_U\rightarrow\mathbb{R}/2\pi\mathbb{Z},$$
so that we have $d\theta_e|_{M_b}=\gamma_e(b)$,
$d\theta_m|_{M_b}=\gamma_m(b)$. Note that $\theta_e$ can be extended
to a global function on $M$, while $\theta_m$ cannot because of the
nontrivial monodromy $\theta_m\mapsto\theta_m+\theta_e$.

To construct a hyperk\"{a}hler metric on $M$, we define a
homomorphism $Z:\Gamma\rightarrow\mathbb{C}$ by setting
\begin{eqnarray*}
Z(\gamma_e(b)) & = & b,\\
Z(\gamma_m(b)) & = & \frac{1}{2\pi i}(b\log\frac{b}{r}-b).
\end{eqnarray*}
$Z$ is called the \textit{central charge} in the physics literature.
The functions $Z_e:=Z(\gamma_e(b))$, $Z_m:=Z(\gamma_m(b))$ are
defined in this way so that they are compatible with the monodromy
of $\theta_e,\theta_m$ respectively. Then, we can define two
families of $\mathbb{C}^*$-valued functions
$\chi_e^\textrm{sf}(\zeta),\chi_m^\textrm{sf}(\zeta)$ locally on
$M$:
\begin{eqnarray*}
\chi_e^\textrm{sf}(\zeta) & = &
\exp\Bigg[\frac{\pi}{\epsilon}(\zeta^{-1}Z_e+\zeta\bar{Z}_e)+i\theta_e\Bigg],\\
\chi_m^\textrm{sf}(\zeta) & = &
\exp\Bigg[\frac{\pi}{\epsilon}(\zeta^{-1}Z_m+\zeta\bar{Z}_m)+i\theta_m\Bigg],
\end{eqnarray*}
parameterized by $\zeta\in\mathbb{C}^*$. Here, $\epsilon>0$ is a
constant. The functions
$\chi_e^\textrm{sf}(\zeta),\chi_m^\textrm{sf}(\zeta)$ give the
so-called \textit{semi-flat} local coordinates on $M$. Notice that
the coordinate $\chi_e^\textrm{sf}(\zeta)$ extends to a global
holomorphic function on $M$, while $\chi_m^\textrm{sf}(\zeta)$ has
nontrivial monodromy around $0\in B$ given by
$\chi_m^\textrm{sf}(\zeta)\mapsto\chi_e^\textrm{sf}(\zeta)\chi_m^\textrm{sf}(\zeta)$.

Now, consider the two-forms
$$\Omega^\textrm{sf}(\zeta)=\frac{d\chi_e^\textrm{sf}(\zeta)}{\chi_e^\textrm{sf}(\zeta)}
\wedge\frac{d\chi_m^\textrm{sf}(\zeta)}{\chi_m^\textrm{sf}(\zeta)},\
\zeta\in\mathbb{C}^*$$ on $M$.\footnote{The definitions of the
holomorphic two-forms $\Omega^\textrm{sf}(\zeta)$ and
$\Omega(\zeta)$ here differ from those in \cite{GMN08} by
multiplication by the constant $-\epsilon/4\pi^2$.} In \cite{GMN08},
it was checked that this family of two-forms
$\{\Omega^\textrm{sf}(\zeta):\zeta\in\mathbb{C}^*\}$ satisfies all
the hypotheses in the theorem of Hitchin et al \cite{HKLR87},
\cite{Hitchin92} and concluded that $M\times\mathbb{C}P^1$ equipped
with $\{\Omega^\textrm{sf}(\zeta):\zeta\in\mathbb{C}^*\}$ is the
\textit{twistor space} of a hyperk\"{a}hler metric $g^\textrm{sf}$
on $M$. However, since $\chi_m^\textrm{sf}(\zeta)$ is not globally
defined on $M$, this semi-flat metric $g^\textrm{sf}$ is singular at
a point over $b=0\in B$.

To obtain a smooth hyperk\"{a}hler metric on $M$, Gaiotto, Moore and
Neitzke argued that we should modify the function
$\chi_m^\textrm{sf}(\zeta)$ by instanton corrections. (We need not
correct the function $\chi_e^\textrm{sf}(\zeta)$ and thus we shall
set $\chi_e(\zeta)=\chi_e^\textrm{sf}(\zeta)$.) They did so by
solving a Riemann-Hilbert problem which is described as follows.
Consider the following rays in the $\zeta$-plane.
\begin{eqnarray*}
l_+ & = & \{\zeta\in\mathbb{C}^*:b/\zeta\in\mathbb{R}_{<0}\},\\
l_- & = & \{\zeta\in\mathbb{C}^*:b/\zeta\in\mathbb{R}_{>0}\}.
\end{eqnarray*}
These are called the \textit{BPS rays} corresponding to the central
charge $Z_e$. The Riemann-Hilbert problem then asks for a family of
holomorphic functions $\{\chi_m(\zeta):\zeta\in\mathbb{C}^*\}$ on
$M$, which are piecewise holomorphic in $\zeta\in\mathbb{C}^*$, such
that the following two conditions are satisfied.\footnote{See
Section 4.4 in \cite{GMN08} for details; due to the choice of the
monodromy, our formulas (\ref{wc+}), (\ref{wc-}) differ from the
formulas (4.52a), (4.52b) on p.16 of \cite{GMN08} by a sign.}
\begin{enumerate}
\item[(a)]
$\chi_m(\zeta)$ is discontinuous across the BPS rays $l_\pm$ in the
following way: Let $(\chi_m(\zeta))^+_{l_+}$,
$(\chi_m(\zeta))^-_{l_+}$ be the limit of $\chi_m(\zeta)$ as $\zeta$
approaches $l_+$ in the clockwise and counter-clockwise direction
respectively, and similarly, $(\chi_m(\zeta))^+_{l_-}$,
$(\chi_m(\zeta))^-_{l_-}$ be the limit of $\chi_m(\zeta)$ as $\zeta$
approaches $l_-$ in the clockwise and counter-clockwise direction
respectively. Then we require that
\begin{eqnarray}\label{wc+}
(\chi_m(\zeta))^-_{l_+} & = &
(\chi_m(\zeta))^+_{l_+}(1+\chi_e(\zeta)),\\
(\chi_m(\zeta))^+_{l_-} & = &
(\chi_m(\zeta))^-_{l_-}(1+\chi_e^{-1}(\zeta)).\label{wc-}
\end{eqnarray}

\item[(b)]
Let
$$\Upsilon(\zeta)=\chi_m(\zeta)\exp\Bigg[-\frac{\pi}{\epsilon}(\zeta^{-1}Z_m+\zeta\bar{Z}_m)\Bigg].$$
Then we require that the limit of $\Upsilon(\zeta)$ as
$\zeta\rightarrow0$ and $\zeta\rightarrow\infty$ exists, and the
limits are related by
$$\lim_{\zeta\rightarrow0}\Upsilon(\zeta)=\overline{\lim_{\zeta\rightarrow\infty}\Upsilon(\zeta)}.$$
\end{enumerate}
It is ingenious that Gaiotto, Moore and Neitzke were able to write
down the following beautiful and explicit formula for
$\chi_m(\zeta)$ in \cite{GMN08}.\footnote{The generalization of this
formula, which is an integral equation satisfied by the functions
$\chi_\gamma(\zeta)$, turns out to be the key in the general
construction of hyperk\"{a}hler metrics on general complex
integrable systems. In particular, one can obtain successive
approximations of the desired hyperk\"{a}hler metric by iteratively
solving the integral equation.}
\begin{eqnarray*}
\chi_m(\zeta) & = & \chi_m^\textrm{sf}(\zeta)\exp\frac{i}{4\pi}
\Bigg[\int_{l_+}\log(1+\chi_e(\zeta'))\frac{\zeta'+\zeta}{\zeta'-\zeta}\frac{d\zeta'}{\zeta'}\\
&   & \qquad\qquad\qquad\qquad
-\int_{l_-}\log(1+\chi_e(\zeta')^{-1})\frac{\zeta'+\zeta}{\zeta'-\zeta}\frac{d\zeta'}{\zeta'}\Bigg].
\end{eqnarray*} Now, the family of two forms
$$\Omega(\zeta)=\frac{d\chi_e(\zeta)}{\chi_e(\zeta)}\wedge\frac{d\chi_m(\zeta)}{\chi_m(\zeta)}$$
on $M$ again satisfies the hypotheses the theorem of Hitchin et al,
and hence defines a smooth hyperk\"{a}hler metric $g$ on $M$ (which
can be determined explicitly from the family of two-forms
$\Omega(\zeta)$, $\zeta\in\mathbb{C}^*$). Furthermore, Gaiotto,
Moore and Neitzke verified that this is nothing but the Ooguri-Vafa
metric constructed by the Gibbons-Hawking ansatz \cite{OV96},
\cite{GW00}.

\section{Holomorphic discs, wall-crossing and mirror symmetry}\label{sec3}

In this section, we study mirror symmetry for the Ooguri-Vafa metric
from the viewpoint of the SYZ Conjecture \cite{SYZ96} and interpret
the formulas (\ref{wc+}), (\ref{wc-}), which describe the
discontinuity of the function $\chi_m(\zeta)$ across the BPS rays
$l_\pm$, as wall-crossing formulas which appear in the SYZ
construction of the instanton-corrected mirror manifold, following
the approach of Auroux (see Section 5 in \cite{Auroux07} and Section
3 in \cite{Auroux09}). These wall-crossing phenomena are special
cases of those studied first by Kontsevich and Soibelman in
\cite{KS04}, which also played a crucial role in the foundational
work of Gross and Siebert \cite{GS07}.\\

To begin with, recall that we have a family of two-forms
$\{\Omega(\zeta):\zeta\in\mathbb{C}^*\}$ on $M$. For each
$\zeta\in\mathbb{C}^*$, $\Omega(\zeta)$ is holomorphic with respect
to a complex structure $J(\zeta)$, and there is a corresponding
K\"{a}hler form $\omega(\zeta)$. We want to write down a formula for
$\omega(\zeta)$. To do this, recall that, in the Gibbons-Hawking
ansatz, the hyperk\"{a}hler metric $g$ on $M$ is determined by a
triplet of symplectic forms
\begin{eqnarray*}
\omega_1 & = & db_1\wedge\alpha+Vdb_2\wedge db_3,\\
\omega_2 & = & db_2\wedge\alpha+Vdb_3\wedge db_1,\\
\omega_3 & = & db_3\wedge\alpha+Vdb_1\wedge db_2,
\end{eqnarray*}
where $b=b_1+ib_2\in B$,
$b_3=\frac{\epsilon\theta_e}{2\pi}\in\mathbb{R}/\epsilon\mathbb{Z}$,
$V=V(b_1,b_2,b_3)$ is a positive harmonic function on
$(B\times\mathbb{R}\setminus\{0\}\times\epsilon\mathbb{Z})/\epsilon\mathbb{Z}$,
and $\alpha$ is a connection one-form on $M$ (which can be realized
as a partial compactification of a circle bundle over
$(B\times\mathbb{R}\setminus\{0\}\times\epsilon\mathbb{Z})/\epsilon\mathbb{Z}$)
of the form $\frac{d\theta_m}{2\pi}+A(b_1,b_2,b_3)$ which satisfies
$d\alpha=dA=\star dV$. There are explicit formulas for $V$ and
$\alpha$, see Remark~\ref{rmk3.1}. The symplectic form
$\omega(\zeta)$, which is K\"{a}hler with respect to $J(\zeta)$, is
then given by
$$\omega(\zeta)=\frac{4\pi^2}{\epsilon}\Bigg[\frac{i(\bar{\zeta}\omega_+-\zeta\omega_-)
+(1-|\zeta|^2)\omega_3}{1+|\zeta|^2}\Bigg],$$ where
$\omega_\pm=\omega_1\pm i\omega_2$. We also have
\begin{equation}\label{Omega}
\Omega(\zeta)=-\frac{4\pi^2}{\epsilon}\Bigg[\frac{1}{2i}(\zeta^{-1}\omega_++\zeta\omega_-)+\omega_3\Bigg].
\end{equation}

Now, we shall fix $\zeta\in\mathbb{C}^*$ and denote by $M(\zeta)$
the manifold $M$ equipped with the K\"{a}hler form $\omega(\zeta)$
and the holomorphic two-form $\Omega(\zeta)$. We want to study the
SYZ mirror symmetry for $M(\zeta)$. The first step is to construct a
special Lagrangian torus fibration. Consider the $S^1$-action on $M$
given by rotating the angle coordinate $\theta_m$:
$$e^{it}\cdot(b_1,b_2,\theta_e,\theta_m)=(b_1,b_2,\theta_e,\theta_m+t).$$
\begin{lem}
This $S^1$-action is Hamiltonian with respect to $\omega(\zeta)$
when $|\zeta|=1$, and the moment map is then given by
$$\mu_{S^1}=\frac{2\pi}{\epsilon}\textrm{Im}(\bar\zeta b):M\rightarrow\mathbb{R}.$$
\end{lem}
\begin{proof}
It is clear that the $S^1$-action preserves $\omega(\zeta)$. By a
straightforward computation, we have
\begin{eqnarray*}
\omega(\zeta) & = &
\frac{4\pi^2}{\epsilon}d\Bigg[\frac{-2\textrm{Im}(\bar{\zeta}b)+(1-|\zeta|^2)\frac{\epsilon\theta_e}{2\pi}}{1+|\zeta|^2}
\Bigg]\wedge\Bigg[\frac{d\theta_m}{2\pi}+A\Bigg]\\
& &
\qquad\qquad+\frac{4\pi^2}{\epsilon}Vd\textrm{Re}(\bar{\zeta}b)\wedge
d\Bigg[\frac{\frac{\epsilon\theta_e}{\pi}+(1-|\zeta|^2)\textrm{Im}(\bar{\zeta}b)}{1+|\zeta|^2}\Bigg].
\end{eqnarray*}
Hence,
$$\iota_{\frac{\partial}{\partial\theta_m}}\omega(\zeta)=\frac{2\pi}{\epsilon}
d\Bigg[\frac{2\textrm{Im}(\bar{\zeta}b)-(1-|\zeta|^2)\frac{\epsilon\theta_e}{2\pi}}{1+|\zeta|^2}\Bigg],$$
which is exact when $|\zeta|=1$, and the moment map is given by
$$\mu_{S^1}=\frac{2\pi}{\epsilon}\textrm{Im}(\bar{\zeta}b).$$
\end{proof}
In view of the above lemma, we shall from now on fix a $\zeta$ such
that $|\zeta|=1$.

Recall that we have a globally defined coordinate
$$\chi_e(\zeta)=\exp\Bigg[\frac{2\pi}{\epsilon}\textrm{Re}(\bar{\zeta}b)+i\theta_e\Bigg]:M\rightarrow\mathbb{C}^*,$$
which is holomorphic with respect to the complex structure
$J(\zeta)$.
\begin{defn}
For $(s,\lambda)\in\mathbb{R}^2$, define
$$T_{s,\lambda}=\{(b_1,b_2,\theta_e,\theta_m)\in M:\log|\chi_e(\zeta)|=s,\mu_{S^1}=\lambda\}.$$
\end{defn}
For $(s,\lambda)\neq(0,0)$, $T_{s,\lambda}$ is a torus embedded in
$M$, and $T_{0,0}$ is nodal. Now, the reduced space
$M_{red,\lambda}=\mu_{S^1}^{-1}(\lambda)/S^1$ is topologically an
annulus, and from formula (\ref{Omega}), we can see that the reduced
holomorphic volume form is given by
$$\Omega(\zeta)_{red,\lambda}=\iota_{\frac{\partial}{\partial\theta_m}}\Omega(\zeta)
=-id\log\chi_e(\zeta).$$ Thus, by Theorem 1.2 in Gross
\cite{Gross00}, we have the following result (see also Proposition
5.2 in Auroux \cite{Auroux07}).
\begin{prop}
Each $T_{s,\lambda}$ is special Lagrangian in $M$ with respect to
$\omega(\zeta),\Omega(\zeta)$. Hence, the map
$\Psi:M(\zeta)\rightarrow\mathbb{R}^2$ defined by
$$\Psi=(\log|\chi_e(\zeta)|,\mu_{S^1})$$
gives a special Lagrangian torus fibration, with a single nodal
fiber $T_{0,0}$.
\end{prop}
In fact, we have
$$\log|\chi_e(\zeta)|=\frac{2\pi}{\epsilon}\textrm{Re}(\bar{\zeta}b),\
\mu_{S^1}=\frac{2\pi}{\epsilon}\textrm{Im}(\bar{\zeta}b),$$ and thus
$$\Psi=\frac{2\pi\bar{\zeta}}{\epsilon}\psi,$$
where $\psi:M\rightarrow B$ is the elliptic fibration that we start
with. So the image of $\Psi$ is given by
$\frac{2\pi}{\epsilon}B=\{b\in\mathbb{C}:|b|<\frac{2\pi
r}{\epsilon}\}$. We will abuse notations and use $B$ to denote
$\{b\in\mathbb{C}:|b|<\frac{2\pi r}{\epsilon}\}$.

Now, as the base of a Lagrangian torus fibration, $B$ is a
two-dimensional affine manifold with a unique singular point at
$b=0\in B$. This is called the \textit{focus-focus singularity} in
Hamiltonian mechanics (see for example Section 3 in Casta\~{n}o
Bernard-Matessi \cite{CBM06}). There are symplectic affine
coordinates on $B$ defined as follows (see Hitchin \cite{Hitchin97}
for details). First let $\{\gamma_e^*,\gamma_m^*\}$ be the basis of
$H_1(T_{s,\lambda},\mathbb{Z})$ dual to
$\{\gamma_e,\gamma_m\}\subset H^1(T_{s,\lambda},\mathbb{Z})$. For
every tangent vector $\nu$ on $B$, lift it to a normal vector field
(which we again denoted by $\nu$) on $T_{s,\lambda}$. Then the
1-forms
$$\omega_e(\zeta)(\nu)=\int_{\gamma_e^*}\iota_\nu\omega(\zeta),\
\omega_m(\zeta)(\nu)=\int_{\gamma_m^*}\iota_\nu\omega(\zeta),$$ on
$B$ are closed, and thus there are locally defined coordinates
$\phi_e(\zeta)$, $\phi_m(\zeta)$ on $B$ such that
$d\phi_e(\zeta)=\omega_e(\zeta), d\phi_m(\zeta)=\omega_m(\zeta)$.
These are called the \textit{symplectic affine coordinates} on $B$
with respect to the basis $\{\gamma_e^*,\gamma_m^*\}$.
\begin{prop}
The symplectic affine coordinates on $B$ with respect to the basis
$\{\gamma_e^*,\gamma_m^*\}$ are explicitly given by
\begin{eqnarray*}
\phi_m(\zeta) & = & -\frac{2\pi}{\epsilon}\textrm{Im}(\bar{\zeta}b)\\
\phi_e(\zeta) & = &
-\frac{1}{\epsilon}\textrm{Re}\Bigg[\bar{\zeta}(b\log\frac{b}{r}-b)\Bigg].
\end{eqnarray*}
\end{prop}
\begin{proof}
Since $|\zeta|=1$, we have
$$\omega(\zeta)=-\frac{4\pi^2}{\epsilon}\textrm{Im}(\bar\zeta\omega_+)=-\frac{4\pi^2}{\epsilon}
d\textrm{Im}(\bar\zeta b)\wedge(\frac{d\theta_m}{2\pi}+A)+2\pi
Vd\textrm{Re}(\bar\zeta b)\wedge d\theta_e.$$ As $V=V(b_1,b_2,b_3)$
and $A=A(b_1,b_2,b_3)$ are independent of $\theta_m$ (recall that
$b_3=\epsilon\theta_e/2\pi$), it is easy to see that
$$d\phi_m(\zeta)=\omega_m(\zeta)=-\frac{2\pi}{\epsilon}d\textrm{Im}(\bar\zeta b).$$
Hence, we can take $\phi_m(\zeta)=
-\frac{2\pi}{\epsilon}\textrm{Im}(\bar{\zeta}b)$.

On the other hand, as will be seen in Remark~\ref{rmk3.1}, we can
decompose $V$ and $A$ into sums of semi-flat and instanton parts,
i.e. $V=V^\textrm{sf}+V^\textrm{inst},\
A=A^\textrm{sf}+A^\textrm{inst}$. And observe that both
$V^\textrm{inst}$ and $A^\textrm{inst}$ are periodic in $\theta_e$,
so we have
$$\int_{\gamma_e^*}\iota_\nu\omega(\zeta)=\int_{\gamma_e^*}\iota_\nu
\Bigg(-\frac{4\pi^2}{\epsilon}d\textrm{Im}(\bar\zeta b)\wedge
A^\textrm{sf}+2\pi V^\textrm{sf}d\textrm{Re}(\bar\zeta b)\wedge
d\theta_e\Bigg).$$ Now, by the explicit formulas for $V^\textrm{sf}$
and $A^\textrm{sf}$ in Remark~\ref{rmk3.1}, we compute
\begin{eqnarray*}
&   & -\frac{4\pi^2}{\epsilon}d\textrm{Im}(\bar\zeta b)\wedge
A^\textrm{sf}+2\pi V^\textrm{sf}d\textrm{Re}(\bar\zeta b)\wedge
d\theta_e\\
& = &
\frac{-i}{2\epsilon}\Bigg(\log\frac{b}{r}-\log\frac{\bar{b}}{r}\Bigg)d\textrm{Im}(\bar\zeta
b)\wedge d\theta_e-\frac{1}{2\epsilon}\Bigg(\log\frac{b}{r}+
\log\frac{\bar{b}}{r}\Bigg)d\textrm{Re}(\bar\zeta b)\wedge
d\theta_e\\
& = &
-\frac{1}{\epsilon}d\textrm{Re}\Bigg[\bar{\zeta}(b\log\frac{b}{r}-b)\Bigg]\wedge
d\theta_e.
\end{eqnarray*}
Hence,
$$d\phi_e(\zeta)=\omega_e(\zeta)=-\frac{1}{\epsilon}d\textrm{Re}\Bigg[\bar{\zeta}(b\log\frac{b}{r}-b)\Bigg],$$
and we can take
$\phi_e(\zeta)=-\frac{1}{\epsilon}\textrm{Re}\big[\bar{\zeta}(b\log\frac{b}{r}-b)\big]$.
\end{proof}
\begin{nb}$\mbox{}$
\begin{enumerate}
\item[(1)] Notice that the symplectic affine coordinates are of the form
stated by Casta\~{n}o Bernard-Matessi on p.511 in \cite{CBM06}, as
expected.
\item[(2)] In the same way, one can show that the complex affine coordinates
(which correspond to periods of the form
$\textrm{Im}(\Omega(\zeta))$), with respect to the basis
$\{-\gamma_e^*,\gamma_m^*\}$, are given by
$$\frac{2\pi}{\epsilon}\textrm{Re}(\bar\zeta
b)=\log|\chi_e(\zeta)|,\
\frac{1}{\epsilon}\textrm{Im}\big[\bar{\zeta}(b\log\frac{b}{r}-b)\big]=\log|\chi_m^\textrm{sf}(\zeta)|.$$
\item[(3)]
The central charge $Z:\Gamma\rightarrow\mathbb{C}$ satisfies the
following relations:
\begin{eqnarray*}
\int_{\gamma_m^*}\omega_+=\frac{1}{2\pi}dZ_e,\
\int_{\gamma_e^*}\omega_+=-\frac{1}{2\pi}dZ_m.
\end{eqnarray*}
If we define $\check{Z}:\Gamma^\vee\rightarrow\mathbb{C}$ by setting
$Z(\gamma_m^*)=Z_e$ and $Z(\gamma_e^*)=-Z_m$, then $\check{Z}$
agrees with the definition of the central charge for a complex
integrable system given by Kontsevich-Soibelman in Section 2.7 in
\cite{KS08}.
\end{enumerate}
\end{nb}

Having constructed a special Lagrangian torus fibration
$\Psi:M(\zeta)\rightarrow B$ and computed the symplectic affine
coordinates on the base $B$, let us recall the construction of the
mirror manifold $\check{M}(\zeta)$ (as a complex manifold) as
suggested by the SYZ Conjecture \cite{SYZ96}. First of all, consider
the moduli space of pairs $(T_{s,\lambda},\nabla)$, where
$T_{s,\lambda}$ is a nonsingular special Lagrangian torus fiber and
$\nabla$ is a flat $U(1)$-connection on the trivial complex line
bundle over $T_{s,\lambda}$. The mirror manifold $\check{M}(\zeta)$
should be, at least topologically, a partial compactification of
this moduli space. More precisely, $\check{M}(\zeta)$ should contain
the quotient $TB'/\Gamma$ of the tangent bundle of
$B'=B\setminus\{0\}$ by the lattice $\Gamma$ (while $M(\zeta)$
contains $T^*B'/\Gamma^\vee$) as a dense open subset. And, given the
symplectic affine coordinates $\phi_m(\zeta)$, $\phi_e(\zeta)$ on
$B'\subset B$, the complex coordinates on
$TB'/\Gamma\subset\check{M}(\zeta)$ are naturally given by
exponentiating the complexified coordinates, so we set
$$w=\exp(\phi_m(\zeta)+i\check{\theta}_m),\
u^\textrm{sf}=\exp(-\phi_e(\zeta)-i\check{\theta}_e).$$ These give
local complex coordinates on the open dense subset $TB'/\Gamma$ in
the mirror manifold $\check{M}(\zeta)$. However, while the
coordinate $w$ is globally defined on $\check{M}(\zeta)$, the other
coordinate $u^\textrm{sf}$ does not extend to a global coordinate
due to nontrivial monodromy around $b=0\in B$: $u^\textrm{sf}\mapsto
u^\textrm{sf}w$.\footnote{Since the monodromy of $\theta_e,\theta_m$
around $b=0\in B$ is given by the matrix $T=\left(\begin{array}{cc}
1 & 1 \\
0 & 1
\end{array}\right)$, the monodromy of the dual coordinates $\check\theta_e$,
$\check\theta_m$ should be given by the matrix
$(T^{-1})^t=\left(\begin{array}{cc}
1 & 0 \\
-1 & 1
\end{array}\right)$, i.e.
$\check\theta_e\mapsto\check\theta_e-\check\theta_m$,
$\check\theta_m\mapsto\check\theta_m$.}

In fact, this is a general phenomenon: When the special Lagrangian
torus fibration $M\rightarrow B$ admits singular fibers ($T_{0,0}$
in our case), the local complex coordinates on the open dense subset
$TB'/\Gamma$ of the mirror $\check{M}\rightarrow B$ given by
exponentiating the complexification of the symplectic affine
coordinates on the smooth part $B'$ of the base $B$ cannot be
extended to the whole mirror manifold $\check{M}$ due to nontrivial
monodromy around the singular locus $\Delta=B\setminus B'$. To
obtain the correct complex coordinates on the mirror manifold, we
must incorporate the information of the singular special Lagrangian
fibers and nontrivial holomorphic discs with boundary on the smooth
special Lagrangian torus fibers (disc instantons). More precisely,
we need to modify the gluing of the local complex charts on the
mirror manifold by disc instanton corrections according to certain
wall-crossing formulas. This approach of constructing the corrected
mirror manifolds was first suggested by Kontsevich and Soibelman in
\cite{KS04} in the two dimensional case (K3 surfaces). Later this
was studied and generalized by Gross and Siebert \cite{GS07} to
higher dimensional cases. Explicit examples which indicate directly
the relation of the gluing formulas to holomorphic discs instantons
were first given by Auroux in \cite{Auroux07}, \cite{Auroux09}.

To carry out the construction of the instanton-corrected mirror
manifold in our case, we shall first determine which special
Lagrangian torus fibers bound nontrivial holomorphic discs. We have
the following proposition (see also Lemma 5.4 in Auroux
\cite{Auroux07}).
\begin{prop}
The special Lagrangian torus $T_{s,\lambda}$ bounds a nontrivial
$J(\zeta)$-holomorphic disc $\varphi:(D^2,\partial
D^2)\rightarrow(M(\zeta),T_{s,\lambda})$ if and only if $s=0$.
\end{prop}
\begin{proof}
Consider the map $f_\zeta=\chi_e(\zeta):M\rightarrow\mathbb{C}^*$.
Note that the image of $f_\zeta$ is the annulus
$\{w\in\mathbb{C}:\exp(-2\pi r/\epsilon)\leq|w|\leq\exp(2\pi
r/\epsilon)\}$.

Now, suppose that $T_{s,\lambda}$ bounds a nontrivial holomorphic
disc $\varphi:(D^2,\partial
D^2)\rightarrow(M(\zeta),T_{s,\lambda})$. Then the composite map
$f_\zeta\circ\varphi:(D^2,\partial
D^2)\rightarrow(\mathbb{C}^*,\{|w|=e^s\})$ is a holomorphic map. By
the maximum principle, $f_\zeta\circ\varphi$ must be a constant map.
Hence, the image of the holomorphic disc is contained in some fiber
of $f_\zeta$. However, for $w\neq1$, the fiber $f_\zeta^{-1}(w)$ is
biholomorphic to an annulus, and so cannot contain any nontrivial
holomorphic disc. So we must have $e^s=1$ or $s=0$.

Conversely, observe that $f_\zeta^{-1}(1)$ is reducible and
biholomorphic to the union of two discs. Thus, for $\lambda\neq0$,
$T_{0,\lambda}$ indeed bounds a nontrivial holomorphic disc which is
contained entirely in the fiber $f_\zeta^{-1}(1)$.
\end{proof}
We remark that, for $\lambda>0$, the special Lagrangian torus
$T_{0,\lambda}$ boounds a nontrivial holomorphic disc with
symplectic area $\lambda$. Denote by $\beta$ the relative homotopy
class of this disc. By deforming $T_{s,\lambda}$ continuously to
$T_{0,\lambda}$ and setting
$$z_\beta(T_{s,\lambda},\nabla)=\exp\Bigg(-\int_\beta\omega(\zeta)\Bigg)\textrm{hol}_\nabla(\partial\beta),$$
we get a globally defined holomorphic function $z_\beta$ on
$\check{M}(\zeta)$ which is nothing but the coordinate $w$ given
above. For $\lambda<0$, the holomorphic disc bounded by
$T_{0,\lambda}$ has area $-\lambda$, and the corresponding
holomorphic coordinate on the mirror is
$z_{-\beta}=z_\beta^{-1}=w^{-1}=\chi_e(-i\zeta)^{-1}$.

We can now construct the instanton-corrected mirror of
$M(\zeta)=(M,\omega(\zeta),\Omega(\zeta))$, following the approach
of Auroux \cite{Auroux07}, \cite{Auroux09}. By the above
proposition, we know that wall-crossing occurs at the wall $\{b\in
B:\textrm{Re}(\bar{\zeta}b)=0\}$. We remark that if we use the
complex affine coordinates on $B$, then the wall is the straight
line in $B$ invariant under monodromy. Now, the wall divides $B$
into two chambers: $B_1$ and $B_2$, as shown in
Figure~\ref{chambers}.

\begin{figure}[htp]
\setlength{\unitlength}{1mm}
\begin{picture}(100,40)
\curve(50,0, 50,4) \curve(50,5, 50,9) \curve(50,10, 50,14)
\curve(50,15, 50,19) \curve(50,20, 50,24) \curve(50,25, 50,29)
\curve(50,30, 50,34) \curve(50,35, 50,39) \put(49,40){$R_+$}
\put(49,-4){$R_-$} \put(30,20){$B_1$}
\put(25,13.5){$u^\textrm{sf}_1=u^\textrm{sf}_2$} \put(65,20){$B_2$}
\put(60,13.5){$u^\textrm{sf}_1=u^\textrm{sf}_2w$}
\put(48.49,18.7){$\times$} \put(43.5,34.5){$u^-_{R_+}$}
\put(51.5,34.5){$u^+_{R_+}$} \put(43.5,3.5){$u^+_{R_-}$}
\put(51.5,3.5){$u^-_{R_-}$}
\end{picture}
\caption{}\label{chambers}
\end{figure}

\noindent On $B\setminus\{b\in B:\textrm{Re}(\bar\zeta b)=0\textrm{
and Im}(\bar\zeta b)\geq0\}$ and $B\setminus\{b\in
B:\textrm{Re}(\bar\zeta b)=0\textrm{ and Im}(\bar\zeta b)\leq0\}$,
we choose different branches of $\log$, say $\log_1$ and $\log_2$,
so that $\log_1=\log_2$ on $B_1$ and $\log_1=\log_2+2\pi i$ on
$B_2$. Denote by
$$\phi_e^k(\zeta)=-\frac{1}{\epsilon}\textrm{Re}\Bigg[\bar{\zeta}(b\log_k\frac{b}{r}-b)\Bigg],\
u^\textrm{sf}_k=\exp(-\phi_e^k(\zeta)-i\check\theta_e)$$ the
coordinates corresponding to the branch $\log_j$, for $j=1,2$. Hence
the gluing of the complex charts of $\check{M}(\zeta)$ defined by
the two sets of coordinates $(w,u^\textrm{sf}_1)$ and
$(w,u^\textrm{sf}_2)$ are given by
\begin{equation*}
\left\{\begin{array}{ll}
u^\textrm{sf}_1=u^\textrm{sf}_2 & \textrm{on $B_1$,}\\
u^\textrm{sf}_1=u^\textrm{sf}_2w & \textrm{on $B_2$,}
\end{array} \right.
\end{equation*}
and this clearly does not define a global holomorphic coordinate.

What we need to do is to modify the gluing across the wall
$\textrm{Re}(\bar{\zeta}b)=0$ by disc instanton corrections as
follows. Consider the rays
\begin{eqnarray*}
R_+ & = & \{b\in B:\textrm{Re}(\bar{\zeta}b)=0\textrm{ and
Im}(\bar{\zeta}b)>0\},\\
R_- & = & \{b\in B:\textrm{Re}(\bar{\zeta}b)=0\textrm{ and
Im}(\bar{\zeta}b)<0\}.
\end{eqnarray*}
Over the chamber $B_1$, let $u^-_{R_+}$ be the coordinate
$u^\textrm{sf}_1=u^\textrm{sf}_2$ as $b\in B$ approaches $R_+$ in
the clockwise direction, and $u^+_{R_-}$ be the coordinate
$u^\textrm{sf}_1=u^\textrm{sf}_2$ as $b\in B$ approaches $R_-$ in
the counter-clockwise direction. Over the chamber $B_2$, let
$u^+_{R_+}$ be the coordinate $u^\textrm{sf}_2$ as $b\in B$
approaches $R_+$ in the counter-clockwise direction, and $u^-_{R_-}$
be the coordinate $u^\textrm{sf}_1$ as $b\in B$ approaches $R_-$ in
the clockwise direction. (See Figure~\ref{chambers}.) The corrected
gluing should then be given by the following wall-crossing formulas.
\begin{eqnarray}\label{mswc+}
u^-_{R_+} & = & u^+_{R_+}(1+w),\\
u^+_{R_-} & = & u^-_{R_-}(1+w^{-1}).\label{mswc-}
\end{eqnarray}
This defines a global holomorphic coordinate on $\check M(\zeta)$.

Now, we claim that the wall-crossing formulas (\ref{mswc+}),
(\ref{mswc-}) can naturally be identified with the formulas
(\ref{wc+}), (\ref{wc-}) which appear in the construction of
Gaiotto, Moore and Neitzke. Indeed, by hyperk\"{a}hler rotation, we
know a priori that the mirror of the Calabi-Yau 2-fold
$M(\zeta)=(M,\omega(\zeta),\Omega(\zeta))$ should be given by
$\check{M}(\zeta)=M(-i\zeta)=(M,\omega(-i\zeta),\Omega(-i\zeta))$.
Also, observe that we have
\begin{eqnarray*}
\log|\chi_e(-i\zeta)| & = &
\frac{2\pi}{\epsilon}\textrm{Re}(\overline{-i\zeta}b)
=-\frac{2\pi}{\epsilon}\textrm{Im}(\bar{\zeta}b)=\phi_m(\zeta),\\
\log|\chi_m^\textrm{sf}(-i\zeta)| & = &
\frac{1}{\epsilon}\textrm{Im}\Bigg[\overline{-i\zeta}(b\log\frac{b}{r}-b)\Bigg]
=\frac{1}{\epsilon}\textrm{Re}\Bigg[\bar\zeta(b\log\frac{b}{r}-b)\Bigg]=-\phi_e(\zeta).
\end{eqnarray*}
Hence, the coordinates $w$ and $u^\textrm{sf}$ can naturally be
identified with the semi-flat coordinates $\chi_e(-i\zeta)$ and
$\chi_m^\textrm{sf}(-i\zeta)$ respectively. (More precisely, this
means that we have a canonical fiber-preserving diffeomorphism
$M\rightarrow\check{M},\
(b_1,b_2,\theta_e,\theta_m)\mapsto(b_1,b_2,\check{\theta}_m=\theta_e,\check{\theta}_e=-\theta_m)$
between $M\rightarrow B$ and $\check{M}\rightarrow B$ identifying
the semi-flat local coordinates; see also Remark~\ref{rmk3.1}.)

So the two sets of equations (\ref{wc+}), (\ref{wc-}) and
(\ref{mswc+}), (\ref{mswc-}) are both defining a global holomorphic
coordinate on $M(-i\zeta)$ by correcting the semi-flat coordinate
$u^\textrm{sf}=\chi_m^\textrm{sf}(-i\zeta)$, and the corrections
involving $w=\chi_e(-i\zeta)$ are of the same form. The only
difference is that the BPS rays $l_+,l_-$ lie in the $\zeta$-plane,
while $R_+,R_-$ lie in $B$. However, we notice that the rays $R_+$,
$R_-$ can be rewritten as
\begin{eqnarray*}
R_+ & = & \{b\in B:b/(-i\zeta)\in\mathbb{R}_{<0}\},\\
R_- & = & \{b\in B:b/(-i\zeta)\in\mathbb{R}_{>0}\}.
\end{eqnarray*}
Now, when $b$ approaches $R_+$ in the counter-clockwise direction,
the BPS ray $l_+=\{\zeta':b/\zeta'\in\mathbb{R}_{<0}\}$ is rotating
in the $\zeta$-plane in the counter-clockwise direction and
approaching the \textit{fixed} $-i\zeta$. Equivalently, $-i\zeta$ is
approaching $l_+$ in the clockwise direction. Likewise, when $b$ is
approaching $R_+$ in the clockwise direction, $-i\zeta$ is
approaching $l_+$ in the counter-clockwise direction; and similarly
for $l_-$ and $R_-$. We therefore come to the main conclusion of
this note:\\

\noindent\textit{Suppose that $|\zeta|=1$. Then the equations
(\ref{wc+}), (\ref{wc-}), with $\zeta$ replaced by $-i\zeta$, which
describe the discontinuity of the holomorphic coordinate
$\chi_m(-i\zeta)$ across the BPS rays $l_\pm$, are equivalent to the
wall-crossing formulas (\ref{mswc+}), (\ref{mswc-}) which appear in
the construction of the instanton-corrected mirror of
$M(\zeta)=(M,\omega(\zeta),\Omega(\zeta))$.\\}

\noindent In particular, we now see clearly how disc instanton
corrections (given by nontrivial holomorphic discs with boundary on
special Lagrangian torus fibers) contribute to the construction of
the Ooguri-Vafa metric.

We end this note by a couple of remarks.
\begin{nb}\label{rmk3.2}
As we mentioned in the introduction, in the case of the Ooguri-Vafa
metric, the Kontsevich-Soibelman wall-crossing formula of first kind
is trivial. This is because the wall of first kind is empty and thus
the numerical Donaldson-Thomas invariants, which are given by an
integer-valued function $\Omega:\Gamma\rightarrow\mathbb{Z}$, is
constant.\footnote{Caution: Do not confuse the $\Omega$ here with
the holomorphic two-form $\Omega(\zeta)$.} More precisely, we have,
for all $b\in B$, $\Omega(\gamma_e)=\Omega(-\gamma_e)=1$ and
$\Omega(\gamma)=0$ for any $\gamma\not\in\{\pm\gamma_e\}$. In turn,
this should be interpreted as the fact that only $\pm\gamma_e$, now
regarded as elements in $H_1(\check{M}(\zeta)_b,\mathbb{Z})$, bounds
nontrivial holomorphic discs in $\check{M}(\zeta)$ with boundary on
the dual special Lagrangian torus fibers. This is closely related to
the comment stated in 1.5(2) on p.16 in \cite{KS08}, where
Kontsevich-Soibelman speculated that the numerical Donaldson-Thomas
invariants $\Omega(\gamma)$ should be counting certain holomorphic
discs in $\check{M}(\zeta)$ "near infinity". As pointed out to me by
Yan Soibelman, one interesting question is to interpret the
wall-crossing formulas in terms of $3d$ Calabi-Yau categories.
\end{nb}
\begin{nb}\label{rmk3.1}[The Ooguri-Vafa metric and SYZ mirror transformations]
In \cite{GMN08}, Gaiotto-Moore-Neitzke decomposed the positive
harmonic function $V=V(b_1,b_2,b_3)$ and the 1-form
$A=A(b_1,b_2,b_3)$ into a sum of semi-flat part and instanton part
(see also \cite{OV96}). More precisely, we can write
\begin{eqnarray*}
V & = & V^\textrm{sf}+V^\textrm{inst},\\
A & = & A^\textrm{sf}+A^\textrm{inst},
\end{eqnarray*}
where
\begin{eqnarray*}
V^\textrm{sf} & = &
-\frac{1}{4\pi\epsilon}\Bigg(\log\frac{b}{r}+\log\frac{\bar{b}}{r}\Bigg),\\
V^\textrm{inst} & = &
\frac{1}{2\pi\epsilon}\sum_{n\neq0}K_0\Bigg(\frac{2\pi}{\epsilon}|nb|\Bigg)e^{in\theta_e},\\
A^\textrm{sf} & = &
\frac{i}{8\pi^2}\Bigg(\log\frac{b}{r}-\log\frac{\bar{b}}{r}\Bigg)d\theta_e,\\
A^\textrm{inst} & = &
-\frac{1}{4\pi\epsilon}(\frac{db}{b}-\frac{d\bar{b}}{\bar{b}})\sum_{n\neq0}(\textrm{sgn}\
n)|b|K_1\Bigg(\frac{2\pi}{\epsilon}|nb|\Bigg)e^{in\theta_e},
\end{eqnarray*}
and $K_0,K_1$ are generalized Bessel functions. Accordingly, we can
decompose the holomorphic two form $\Omega(\zeta)$ and the
symplectic form $\omega(\zeta)$ into a sum of semi-flat and
instanton parts.
\begin{eqnarray*}
\Omega(\zeta) & = & \Omega^\textrm{sf}(\zeta)+\Omega^\textrm{inst}(\zeta),\\
\omega(\zeta) & = &
\omega^\textrm{sf}(\zeta)+\omega^\textrm{inst}(\zeta).
\end{eqnarray*}
It is straightforward to show that we have
$$\Omega^\textrm{sf}(\zeta)=\frac{d\chi_e^\textrm{sf}(\zeta)}{\chi_e^\textrm{sf}(\zeta)}
\wedge\frac{d\chi_m^\textrm{sf}(\zeta)}{\chi_m^\textrm{sf}(\zeta)},$$
which agrees with the formula in Section~\ref{sec2}.

In the case $|\zeta|=1$, we compute the semi-flat parts
$\Omega^\textrm{sf}(\zeta)$ and $\omega^\textrm{sf}(\zeta)$, and
they are respectively given by
\begin{eqnarray*}
\Omega^\textrm{sf}(\zeta) & = &
\Bigg(\frac{2\pi}{\epsilon}d\textrm{Re}(\bar{\zeta}b)+id\theta_e\Bigg)\wedge\Bigg(\frac{2\pi}{\epsilon}d
\textrm{Re}(\bar{\zeta}Z_m)+id\theta_m\Bigg),\\
\omega^\textrm{sf}(\zeta) & = &
-\frac{2\pi}{\epsilon}d\textrm{Im}(\bar{\zeta}b)\wedge
d\theta_m+\frac{2\pi}{\epsilon}d\textrm{Im}(\bar{\zeta}Z_m)\wedge
d\theta_e,
\end{eqnarray*}
Now, if we set
$\check{\theta}_e=-\theta_m,\check{\theta}_m=\theta_e$ (they are the
dual fiber coordinates on $\check{M}=M$), then
\begin{eqnarray*}
\omega^\textrm{sf}(-i\zeta) & = &
\frac{2\pi}{\epsilon}d\textrm{Im}(i\bar{\zeta}b)\wedge
d(-\theta_m)+\frac{2\pi}{\epsilon}d\textrm{Im}(i\bar{\zeta}Z_m)\wedge
d\theta_e\\
& = & \frac{2\pi}{\epsilon}d\textrm{Re}(\bar{\zeta}b)\wedge
d\check{\theta}_e+\frac{2\pi}{\epsilon}d\textrm{Re}(\bar{\zeta}Z_m)\wedge
d\check{\theta}_m.
\end{eqnarray*}
We can then show that
\begin{eqnarray*}
\mathcal{F}^\textrm{sf}(e^{i\omega^\textrm{sf}(-i\zeta)}) & = &
\Omega^\textrm{sf}(\zeta),\\
(\mathcal{F}^\textrm{sf})^{-1}(\Omega^\textrm{sf}(\zeta)) & = &
e^{i\omega^\textrm{sf}(-i\zeta)},
\end{eqnarray*}
where $\mathcal{F}^\textrm{sf}$ is the semi-flat SYZ mirror
transformation introduced in Chan-Leung \cite{Chan-Leung08a},
\cite{Chan-Leung08b}. Moreover, it is easy to see that we have
$$(\mathcal{F}^\textrm{sf})^{-1}(\Omega^\textrm{inst}(\zeta))=0.$$
It is thus very natural to ask whether one can construct an SYZ
mirror transformation $\mathcal{F}$ such that
\begin{eqnarray*}
\mathcal{F}(e^{i\omega(-i\zeta)}) & = & \Omega(\zeta),\\
(\mathcal{F})^{-1}(\Omega(\zeta)) & = & e^{i\omega(-i\zeta)}.
\end{eqnarray*}
This is related to the question of writing down the K\"{a}hler
structure on the mirror manifold in terms of the holomorphic volume
form on the original manifold. We hope to return to this in a later
paper.
\end{nb}


\begin{thebibliography}{99}

\bibitem{Auroux07}
D. Auroux, \textit{Mirror symmetry and T-duality in the complement
of an anticanonical divisor}. J. Gokova Geom. Topol. GGT, {\bf 1}
(2007), 51--91 (arXiv:0706.3207).

\bibitem{Auroux09}
\underline{\qquad\quad}, \textit{Special Lagrangian fibrations,
wall-crossing, and mirror symmetry}. To appear in Surveys in
Differential Geometry (arXiv:0902.1595).

\bibitem{CBM06}
R. Casta\~{n}o Bernard and D. Matessi, \textit{Lagrangian 3-torus
fibrations}. J. Differential Geom. 81 (2009), no. 3, 483--573
(math.SG/0611139).

\bibitem{Chan-Leung08a}
K.-W. Chan and N.-C. Leung, \textit{Mirror symmetry for toric Fano
manifolds via SYZ transformations}. To appear in Adv. Math.
(arXiv:0801.2830).

\bibitem{Chan-Leung08b}
\underline{\qquad\quad}, \textit{On SYZ mirror transformations}. To
appear in Advanced Studies in Pure Mathematics, "New developments in
Algebraic Geometry, Integrable Systems and Mirror Symmetry"
(arXiv:0808.1551).

\bibitem{GMN08}
D. Gaiotto, G. Moore and A. Neitzke, \textit{Four-dimensional
wall-crossing via three-dimensional field theory}. Preprint 2008
(arXiv:0807.4723).

\bibitem{Gross00}
M. Gross, \textit{Examples of special Lagrangian fibrations}.
Symplectic geometry and mirror symmetry (Seoul, 2000), 81--109,
World Sci. Publ., River Edge, NJ, 2001 (math.AG/0012002).

\bibitem{GS07}
M. Gross and B. Siebert, \textit{From real affine geometry to
complex geometry}. Preprint 2007 (math.AG/0703822).

\bibitem{GW00}
M. Gross and P. M. H. Wilson, \textit{Large complex structure limits
of $K3$ surfaces}. J. Differential Geom. 55 (2000), no. 3, 475--546
(math.DG/0008018).

\bibitem{Hitchin92}
N. Hitchin, \textit{Hyper-K\"{a}hler manifolds}. S\'{e}minaire
Bourbaki, Vol. 1991/92. Ast\'{e}risque No. 206 (1992), Exp. No. 748,
3, 137--166.

\bibitem{Hitchin97}
\underline{\qquad\quad}, \textit{The moduli space of special
Lagrangian submanifolds}. Ann. Scuola Norm. Sup. Pisa Cl. Sci. (4)
25 (1997), no. 3-4, 503--515 (dg-ga/9711002).

\bibitem{HKLR87}
N. Hitchin, A. Karlhede, U. Lindstr\"{o}m and M. Ro\v{c}ek,
\textit{Hyper-K\"{a}hler metrics and supersymmetry}. Comm. Math.
Phys. 108 (1987), no. 4, 535--589.

\bibitem{KS00}
M. Kontsevich and Y. Soibelman, \textit{Homological mirror symmetry
and torus fibrations}. Symplectic geometry and mirror symmetry
(Seoul, 2000), 203--263, World Sci. Publ., River Edge, NJ, 2001
(math.SG/0011041).

\bibitem{KS04}
\underline{\qquad\quad}, \textit{Affine structures and
non-Archimedean analytic spaces}. The unity of mathematics,
321--385, Progr. Math., 244, Birkhauser Boston, Boston, MA, 2006
(math.AG/0406564).

\bibitem{KS08}
\underline{\qquad\quad}, \textit{Stability structures, motivic
Donaldson-Thomas invariants and cluster transformations}. Preprint
2008 (arXiv:0811.2435).

\bibitem{OV96}
H. Ooguri and C. Vafa, \textit{Summing up Dirichlet instantons}.
Phys. Rev. Lett. 77 (1996), no. 16, 3296--3298 (hep-th/9608079).

\bibitem{SYZ96}
A. Strominger, S.-T. Yau and E. Zaslow, \textit{Mirror symmetry is
T-duality}. Nuclear Phys. B, {\bf 479} (1996), no. 1-2, 243--259
(hep-th/9606040).

\end{thebibliography}
\end{document}